\documentclass[a4paper,12pt]{article}
\usepackage[utf8]{inputenc}

\usepackage{amsmath}
\usepackage{amsthm}
\usepackage{amssymb}
\usepackage{mathtools}

\def\R{\mathbb{R}}
\def\C{\mathbb C}
\def\Z{\mathbb Z}

\def\P{\mathbb P}

\def\Q{ \mathbb{Q}}
\def\Aff{{\rm Aff}}

\def\M{{\mathcal M}}
\def\F{{\mathcal F}}

\newtheorem{definition}{Definition}
\newtheorem{remark}{Remark}

\newtheorem{theorem}{Theorem} 

\newtheorem{proposition}{Proposition}
\newtheorem{corollary}{Corollary}

\vspace{5mm}

\vspace{5mm}

\begin{document}

\begin{center}
{\LARGE\bf 
Smooth points of the space of plane foliations with a center
}
\\
\vspace{.25in} 
{\large {\sc Lubomir  Gavrilov}}
\footnote{Institut de Math\'{e}matiques de Toulouse, UMR 5219, Universit\'{e}  de Toulouse,  31062 Toulouse, France. {\tt lubomir.gavrilov@math.univ-toulouse.fr} }
{\large {\sc Hossein Movasati}}
\footnote{
Instituto de Matem\'atica Pura e Aplicada, IMPA, Estrada Dona Castorina, 110, 22460-320, Rio de Janeiro, RJ, Brazil,
{\tt hossein@impa.br}}
\end{center}

\begin{abstract}
We prove that a logarithmic foliation corresponding to a generic line arrangement of 
$d+1 \geq 3$ lines  in the complex plane, with pairwise natural and co-prime   residues, is a smooth point of the center set of plane foliations (vector fields) of degree $d$.  
\end{abstract}
\section{Introduction}

The present paper is a contribution to the classical center-focus problem (the problem of distinguishing between a center and a focus of a plane vector field). We consider the set of complex polynomial plane vector fields of degree at most $d$, or equivalently, affine polynomial degree $d$ foliations in $\C^2$:
    \begin{align*}
\F(d) = \left\{ \F(P(x,y)dy-Q(x,y)dx)\mid  P,Q \in \mathbb C[x,y],\ \deg(P),\deg(Q)\leq d\right\}.
    \end{align*}
We identify $\F(d)$   to the set of coefficients of the polynomials $P,Q$ (that is to say to $\C^{(d+1)(d+2)}$. 
We say that a given foliation (a point in $\F(d)$ ) has a Morse center at $p\in \C^2$, or simply a center, if it allows a local analytic first integral which has a Morse critical point at $p$. It is well-known that the Zariski closure of the set of foliations with a Morse center, the so called \emph{center set}, is an algebraic set, see \cite{LinsNeto2014, mov0}.
 We denote this center set by $\M(d) \subset \F(d)$. It has a canonical decomposition (up to a permutation) 
\begin{align}
\label{component}
\M(d)  = \cup_{i} \mathcal {\bar L}_i, \;\; \mathcal {\bar L}_i  \not \subseteq  \mathcal {\bar L}_j, i\neq j
\end{align}
into closed  irreducible algebraic varieties $\mathcal {\bar L}_i$. \emph{The center-focus problem in this setting is to describe the irreducible components   $  \mathcal {\bar  L}_i$  of the center set $\M(d)$. }
The problem is largely open, except in the quadratic case ($d=2$). It follows from the Dulac's computation   of quadratic systems with a center \cite{Dulac1923}, that  $\M(2)$ has four irreducible components, parameterised via their explicit first integrals.
In the case  $d>2$ only some irreducible components of $\M(d)$ are known. For a conjecturally complete list of cubic systems with a center we refer the reader to \cite{zola94b, zola94a,boja10,both07}.

Suppose that $\mathcal L \subset \M(d)$ is an irreducible algebraic set (algebraic variety) formed by foliations with a center. 
To show that  its Zariski closure $\mathcal {\bar L}$ is also an irreducible component of $\M(d)$, like in (\ref{component}), is a local problem. Therefore
we may choose a suitable point $\F_0  \in \mathcal L \subset \M(d)$ and compare the tangent space of $\mathcal L $ at $\F_0$ and the tangent space of $ \M(d)$ at $\F_0$. If the dimension of these spaces are the same, then the condition $ \mathcal L \not\subseteq  \mathcal L_j$ (\ref{component}) is certainly satisfied and therefore $\mathcal L$ is an irreducible component of the center set $ \M(d)$.

The computation of the tangent cone of $  \M(d)$ (even if $ \M(d)$ is not known!)
turns out to be possible by making use of the machinery of Melnikov functions, as shown by Ilyashenko \cite{il69} (in the Hamiltonian case), Movasati \cite{mov0,mov} (the case of logarithmic foliations), Zare \cite{Zare2019} (pull back foliations), Gavrilov \cite{gavr20} (centers of Abel equations). In all these cases it has been shown, that the corresponding irreducible algebraic set of systems with a center is indeed an irreducible component of $\M(d)$.

In the present paper we focus our attention to logarithmic foliations  of the form
\begin{equation}
\label{omicron2021}
\F_0 : \ \  l_1l_2\dots l_{d+1} \left( \sum_{i=1}^{d+1}\lambda_i\frac{dl_i}{l_i}\right)=0  , \ \ d\geq 2, 
\end{equation}
where $\lambda_i\in\C^*$ and $l_i=l_i(x,y)$ are  complex bivariate polynomials of degree one.
 Obviously the foliation $\F_0$ has a first integral of the form
\begin{align}
\label{first}
 l_1^{\lambda_1} l_2^{\lambda_2} \dots l_{d+1}^{\lambda_{d+1}}. 
\end{align}
In what follows we suppose that  the polynomials $l_i$ define a line arrangement without triple intersection points (a general line arrangement), and that $\lambda_i\neq 0$. The set of such foliations is denoted by ${\mathcal L}(1^{d+1})={\mathcal L}(1,1,\ldots,1)$. 
The Zariski closure ${\bar {\mathcal L}}(1^{d+1})\subset \F(d)$ is an irreducible component of the corresponding center set $ \M(d)$  \cite{mov}. If another irreducible component of $\M(d)$ is of dimension at least equal to the co-dimension of 
${\mathcal L}(1^{d+1})$, then it certainly intersects ${\bar {\mathcal L}}(1^{d+1})$. Therefore, the study of the structure of the center set in a small neighbourhood of 
${\bar {\mathcal L}}(1^{d+1})$ implies also a global information on $ \M(d)$. Note that if the foliation $\F_0$ belongs to the intersection of 
${\mathcal L}(1^{d+1})$ with another irreducible component of the center set, then $\F_0$ is a non smooth point  of $ \M(d)$. 
This motivates the following problem, which is partially solved in the paper : 
\emph{Classify the smooth points of $\M(d)$ along the irreducible component ${\mathcal L}(1^{d+1})$.}
We prove the following
\begin{theorem}
\label{main1}
Let $\lambda_i$, $i=1,\dots,d+1$, be mutually prime distinct natural numbers. Let $l_i=l_i(x,y)$ , $i=1,\dots,d+1$, be linear bivariate polynomials defining a generic line arrangement (generic means that there are no triple points). Then 
the logarithmic foliation $\F_0$ defined by (\ref{omicron2021})
is a smooth  point of the center set $\M(d)$.  
\end{theorem}
If $ \F_0$ is a general logarithmic foliation of the form (\refeq{omicron2021}) such that $ \F_0$ is a smooth  point of the center set $ \M(d)$, then obviously every small degree $d$ deformation with a persistent center is also a deformation by logarithmic foliations. 
Therefore the above theorem is close to another classical result which we recall now.
Consider the set $\mathcal L(d+1) \subset \M(d)$ formed by Hamiltonian foliations $\F: dH=0$ where $H$ is an arbitrary degree $d+1$ bivariate polynomial.
Suppose in addition that $H$ is a "Morse plus" polynomial (has only Morse cticical points with distinct critical values). It is proved by Ilyashenko \cite{il69}, that if in a deformation of the Morse plus Hamiltonian foliation $dH$ the center persists, then this deformation is  Hamiltonian too.  The proof implies also   that $\F: dH=0$  is a smooth point of $ \M(d)$.

It is clear that when two irreducible components of $ \M(d)$ intersect at $ \F_0$, then $ \F_0$ is a non-smooth point of $ \M(d)$. It is less known that 
even when $ \F_0$ does not belong to different irreducible components of $ \M(d)$, it can still be a non-smooth point of $ \M(d)$. This happens even in the quadratic case (d=2), for an example see the last section of the paper. 

Our final remark is that it follows from the computation of the tangent cone (which turns out to be a tangent space) 
 that ${\mathcal L}(1^{d+1})$ is an irreducible component of the center set $ \M(d)$. 
This proof is quite different compared to the original proof \cite{mov}, as the tangent cone to  $ \M(d)$ is computed at a smooth pont $\F_0$ (like in \cite{il69}) .

The article is organised in the following way. In Section \ref{8dec2021} we develop  the Picard-Lefschetz theory of the fibration  of the polynomial $ x^ny^m$ where  $n,m$ are natural numbers (not necessarily  coprime). 
 In Section \ref{14dec2021} we study the topology of the fibers  of $f$ 

\begin{equation*}
f= l_1^{n_1} l_2^{n_2} \dots l_{d+1}^{n_{d+1}}. 
\end{equation*} 
where $l_i$  are lines in a general position, and $n_i$ are positive integers without common divisors. 
(we do not suppose that $n_i, n_j$ are relatively prime). 
 As a by-product we get a genus formula for the fibers of $f$.
In Section \ref{02dec2021} we generalise  a theorem due to  A'Campo and Gusein-Zade \cite{acam75,guse74} in the context of a logarithmic foliation defined by the polynomial $f$. In Section \ref{14.12.2021} we compute the orbit of a vanishing cycle under the action of the  monodromy  in the homology bundle of $f$. As a by product, this implies that the orbit of this vanishing cycle contains  the homology of the compactified  fiber. This is the only place where we use the fact that $n_i$'s are pairwise coprime. Summing up all these results leads to the proof of Theorem \ref{main1}, given in the last Section \ref{14/12/2021}. 

The article was written while the first author was visiting the University of Toulouse III. He would like to thank  for the stimulating research atmosphere, as well  for the financial support of CIMI and CNRS.

\section{The Picard-Lefschetz formula of a plane non-isolated singularity }
\label{8dec2021}
The first attempt to describe Picard-Lefschetz theory of fibrations with non-reduced fibers is done in \cite{Clemens69}, however, the main result of this paper Theorem 4.4 is not applicable in our context, and so, we elaborate Theorem \ref{12nov2021} which explicitly describe a kind of Picard-Lefschetz formula. 

In this section we consider the local fibration $f:(\C^2,0)\to(\C,0)$ given by $f=x^my^n$, where $m,n$ are two positive integers. We will use the notation:
$$
e:=(m,n),\ p:=\frac{m}{e}, q:=\frac{n}{e} 
$$
where $(m,n)={\rm gcd}(m,n)$ means the greatest common divisor of $m$ and $n$.
It might be easier to follow the content of the present section for the case $e=1$. For $t\in\R^+$ let 
$$
\Gamma:=\left\{(r,s)\in\R^+\times \R^+, \Big| r^ms^n=t\right\}. $$ 
We consider it as an  oriented path in $f^{-1}(t)$ for increasing $s$ for which we use the letter $\gamma$.
 We consider the following parameterization of the fiber $f^{-1}(t)$ for $t\in\R^+$
\begin{eqnarray}
\label{207EuroMulta2021}
& &\R\times \Gamma\to f^{-1}(t),\ \  h=0,1,2,\ldots,e-1\\ \nonumber 
& &(\theta, r,s)\mapsto (x,y)=(re^{2\pi i (\theta q+\frac{h}{m})}, se^{-2\pi i \theta p}).
\end{eqnarray}
The fiber $f^{-1}(t)$  consists of $e$ cylinders indexed by $h$ and the above parametrization is periodic in $\theta$ with period $1$. 
\begin{definition}\rm
By a straight path in $f^{-1}(t)$ we mean a path which is the image of a path $\alpha$ in $\R\times\Gamma$ under the parameterization \eqref{207EuroMulta2021} and  with the following property: $\alpha$  maps bijectively to its image under the projection $\R\times\Gamma\to\R$. \end{definition}
For simplicity,  we consider the parameters with $|t|<1$ and define $L_t:=f^{-1}(t)\cap B$, where $B$ is the complex square $\{(x,y)\in \C^2 \mid |x|\leq 1,|y|\leq 1\}$. In this way, $L_t$ is a union of $e$ compact cylinders, let us say $L_t=\cup_{h=0}^{e-1}L_{t,h}$. A circle in each cylinder $L_{t,h}$ is  parameterized with fixed $(r,s)$ and  for $(r,s)=(1, |t|^{\frac{1}{n}})$ and $( |t|^{\frac{1}{m}},1)$ we get two circles of its boundary and denote  them by $\delta_{1,h}$ and $\delta_{2,h}$, respectively, and give them a natural orientation coming from $\theta\in [0,1]$ running from $0$ to $1$. We denote by $\delta_h:[0,1]\to L_{t,h}$ the closed path given by  the parameterization \eqref{207EuroMulta2021} and fixed $(r,s)$. This is homotopic to $\delta_{1,h}$ and $\delta_{2,h}$. We also denote by $\gamma_h$ the non closed path in $L_{t,h}$ given by  the parameterization \eqref{207EuroMulta2021} and $\theta=0$. Note that  $\gamma:=\gamma_0$ is the only path from $(1,t^{\frac{1}{n}})$ to $(t^{\frac{1}{m}},1)$ in the real plane $\R^2$.

We consider in $B$ two transversal sections $\Sigma_1:=\{x=1\},\Sigma_2:=\{y=1\}$ to the $x$ and $y$-axis, respectively, and define $\Sigma:=\Sigma_1\cup \Sigma_2$.  The intersections $\{|x|=1\}\cap L_t$ and $\{|y|=1\}\cap L_t$ are  the union of circles $\cup_{h=0}^{e-1}\delta_{i,h}$ for $i=1,2$ respectively,  and they have the following finite subsets 
\begin{eqnarray*}
\Sigma_1\cap L_t &=&
\left \{\zeta_{k,h},\ k=0,1 \cdots,q-1,\  h=0,1,\ldots,e-1\right \}, \\
\Sigma_2\cap L_t &=&
\left \{\xi_{l,h},\ l=0,1, \cdots,p-1,\ h=0,1,\ldots,e-1\right \},
\end{eqnarray*}
where 
\begin{equation}
\label{28nov2021}
\zeta_{k,h}:=(1,t^{\frac{1}{n}}e^{-2\pi i (\frac{km-h}{n} )}),\ \ 
\xi_{l,h}:=(t^{\frac{1}{m}}e^{2\pi i (\frac{ln+h}{m})},1).
\end{equation}
For $\Sigma_1\cap L_t$ we have set $\theta=\frac{mk-h}{[m,n]}$ and for $\Sigma_2\cap L_t$ we have set  $\theta=\frac{nl}{[m,n]}$.
We have a natural action of the multiplicative group of $n$-th (resp. $m$-th) roots of unity on the set $\Sigma_1\cap L_t$ (resp. $\Sigma_2\cap L_t$) which is given by multiplication in the second coordinate. 
\begin{figure}
\begin{center}
\includegraphics[width=0.5\textwidth]{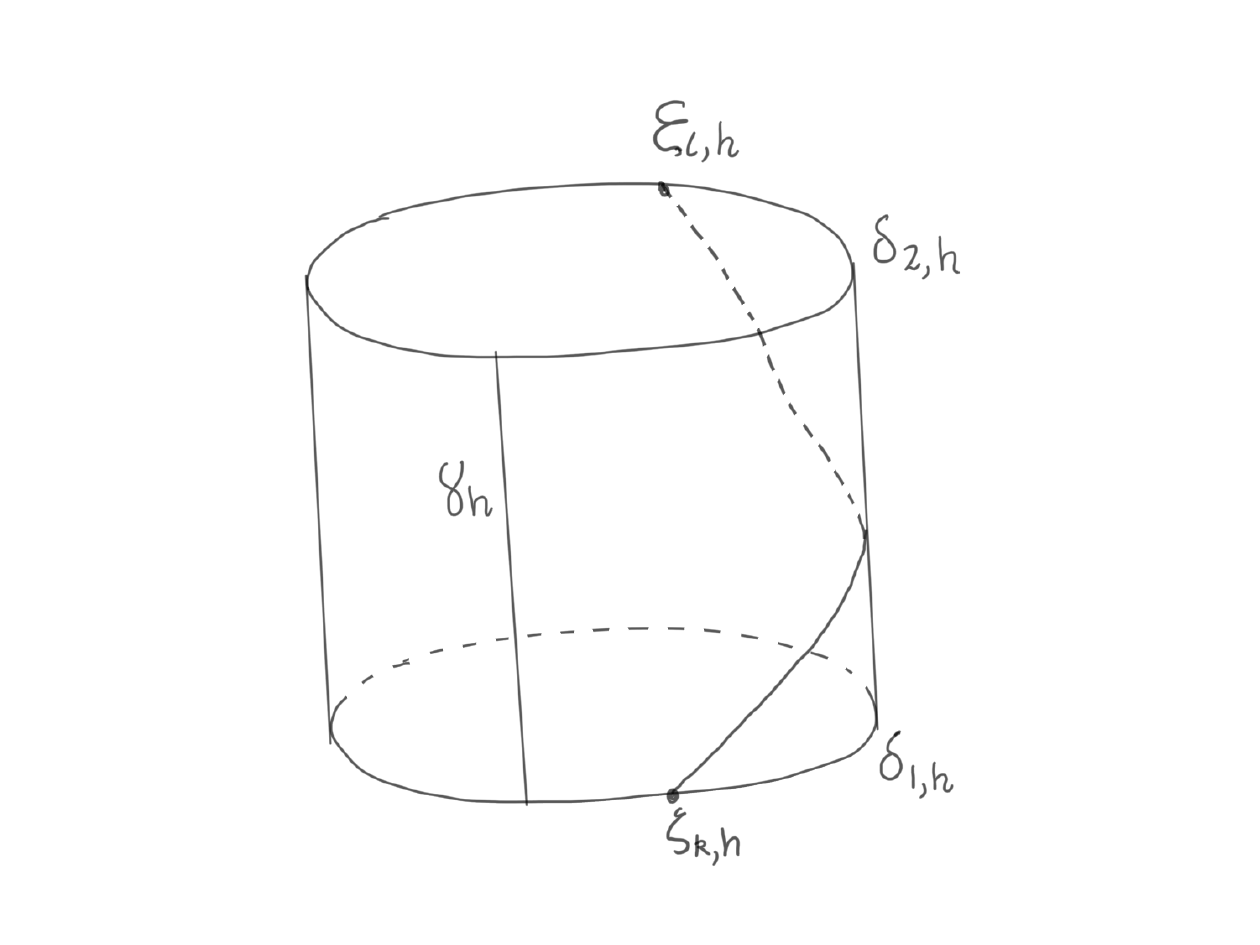}
\caption{A cylinder and straight path}. 
\label{16122021-1}
\end{center}
\end{figure}
\begin{proposition}
\label{10nov2021}
The relative homology group  $H_1(L_t, L_t\cap\Sigma;\Z)$ is freely generated $\Z$-module of rank $n+m$. 
\end{proposition}
\begin{proof}
This follows from the long exact sequence in homology of the pair $L_t, L_t\cap\Sigma$:
$$
\begin{array}{cccccccccccc}
0 &\to & H_1(L_t)&\to& H_1(L_t, \Sigma\cap L_t)&\to& H_0(\Sigma\cap L_t ) &\to & H_0(L_t)&\to& 0& \\
&& \parallel && \parallel
&&
\parallel
&&
\parallel
&&&
\\
&& \Z^e && \Z^{m+n}&& \Z^{m+n} &&
\Z^e
&&&
\\
\end{array}
$$
\end{proof}
Since $p$ and $q$ are coprime positive integers, we can find $a,b\in\Z$ such that 
$$
ap-bq=1,\ \ 0\leq a\leq q-1,\ \ 0\leq b\leq p-1,
$$
for $p,q\geq 2$. Equivalently, $am-bn=e$. We also consider the cases: 
$$
\left\{\begin{array}{ll}
a=1,\ b=0     &  \hbox{ if } p=1 \\
 a=1, b=p-1   &           \hbox{ if } q=1   
\end{array}\right. .
$$
If we change the order of $p$ and $q$ we only need to replace $a$ and $b$ with $q-a$ and $p-b$, respectively.
\begin{theorem}
\label{12nov2021}
Let $\gamma$ be a straight path which connects $\zeta_{k,h}\in L_t\cap \Sigma_1$ to $\zeta_{l,h}\in L_t\cap \Sigma_2$.
The anticlockwise monodromy  $h(\gamma)$ of $\gamma$ around $t=0$ is a straight path which connects
$$
\left\{
\begin{array}{cc}
\zeta_{k,h+1} \hbox{ to } \xi_{l,h+1}    &  \hbox{ if }\ \ \ h+1<e,\\
\zeta_{k-a,0}   \hbox{ to } \xi_{l-b,0}  &  \hbox{ if }\ \ \ h+1=e.
\end{array}
\right. \ \ \ 
$$ 
In particular, we have the classical Picard-Lefschetz formula
\begin{equation}
\label{carcassone2021}
h^{[m,n]}(\gamma)=\gamma+\delta,
\end{equation}
where $[m,n]$ is the lowest common multiple of $m$ and $n$. 
\end{theorem}
\begin{proof}
We consider the differential form $\omega:=m\frac{dx}{x}=-n\frac{dy}{y}$ in $L_t$, where the last equality is written restricted to $L_t$.  We observe that
\begin{eqnarray*}
\int_{\gamma_h}\omega &=& \ln(t),\\
\int_{\delta(\theta)}^{\delta(\theta+\alpha)}\omega &=& 2\pi i [m,n]\alpha, \ \ \hbox{ and hence } \int_{\delta}\omega  =2\pi i [m,n].
\end{eqnarray*}
Actually, in the first formula $\gamma$ can be any path with parametrization in \eqref{207EuroMulta2021} with fixed $\theta$. 
We have
\begin{equation}
\label{24nov2022-1}
\zeta_{k,h}e^{2\pi i\frac{1}{n}}=
(1,e^{-2\pi i (\frac{km-(h+1)}{n} )})=
\left\{
\begin{array}{cc}
\zeta_{k,h+1}     &  \hbox{ if }\ \ \ h+1<e\\
(1,e^{-2\pi i\frac{(k-a)m}{n}})=\zeta_{k-a,0}    &  \hbox{ if }\ \ \ h+1=e
\end{array}
\right. 
\end{equation}
and 
\begin{equation}
\label{24nov2022-2}
\xi_{l,h}e^{2\pi i\frac{1}{m}}=
(e^{-2\pi i (\frac{ln+(h+1)}{m} )},1)=
\left\{
\begin{array}{cc}
\xi_{l,h+1}     &  \hbox{ if }\ \ \  h+1<e\\
(1,e^{-2\pi i\frac{(l-b)n}{m}})=\xi_{l-b,0}    &  \hbox{ if }\ \ \ h+1=e
\end{array}
\right. 
\end{equation}
For the equalities in the case $h+1=e$ we have used $e=am-bn$.  The above equalities imply that $h(\gamma)$ has the right starting and end points as announced in the theorem. 
By Cauchy's  theorem we have 
$$
\int_{\gamma}\omega=\int_{ \gamma_h}\omega+2\pi i[m,n]\frac{nl}{[m,n]}-2\pi i[m,n]\frac{mk-h}{[m,n]}=
\int_{\gamma_h}\omega+2\pi i(nl-mk+h).
$$
Now, we consider a straight path $\tilde \gamma$ in $L_{t,h+1}$ which connects \eqref{24nov2022-1}  to \eqref{24nov2022-2}. A similar formula  for $\tilde\gamma$ as above, and knowing that  $\int_{\gamma_h}\omega=\int_{\gamma_{h+1}}\omega=\ln(t)$ give us  $$
\int_{\tilde \gamma}\omega=
\int_{\gamma}\omega+2\pi i
$$
which implies $h(\gamma)=\tilde \gamma$ for $h+1<e$. For $h+1=e$ this follows from 
$$
\int_{\tilde \gamma}\omega=\int_{ \gamma_0}\omega+2\pi i[m,n]\frac{n(l-b)}{[m,n]}-2\pi i[m,n]\frac{m(k-a)}{[m,n]}=
\int_{\gamma_0}\omega+2\pi i(nl-mk+e).
$$
As a corollary we can get the classical formula for the monodromy $h^{[m, n]}(\gamma)$ in \eqref{carcassone2021}. We know that $h^{e}(\gamma)$ is the straight path connecting $\zeta_{k-a,0}$ to $\xi_{l-b,0}$, and so its $pq$ times iteration is the straight path connecting $\zeta_{k-pqa,0}$ to $\xi_{l-pqb,0}$.  Since $ap-bq=1$ we get  \eqref{carcassone2021}.
\end{proof}
\begin{figure}[t]
\begin{center}
\includegraphics[width=0.5\textwidth]{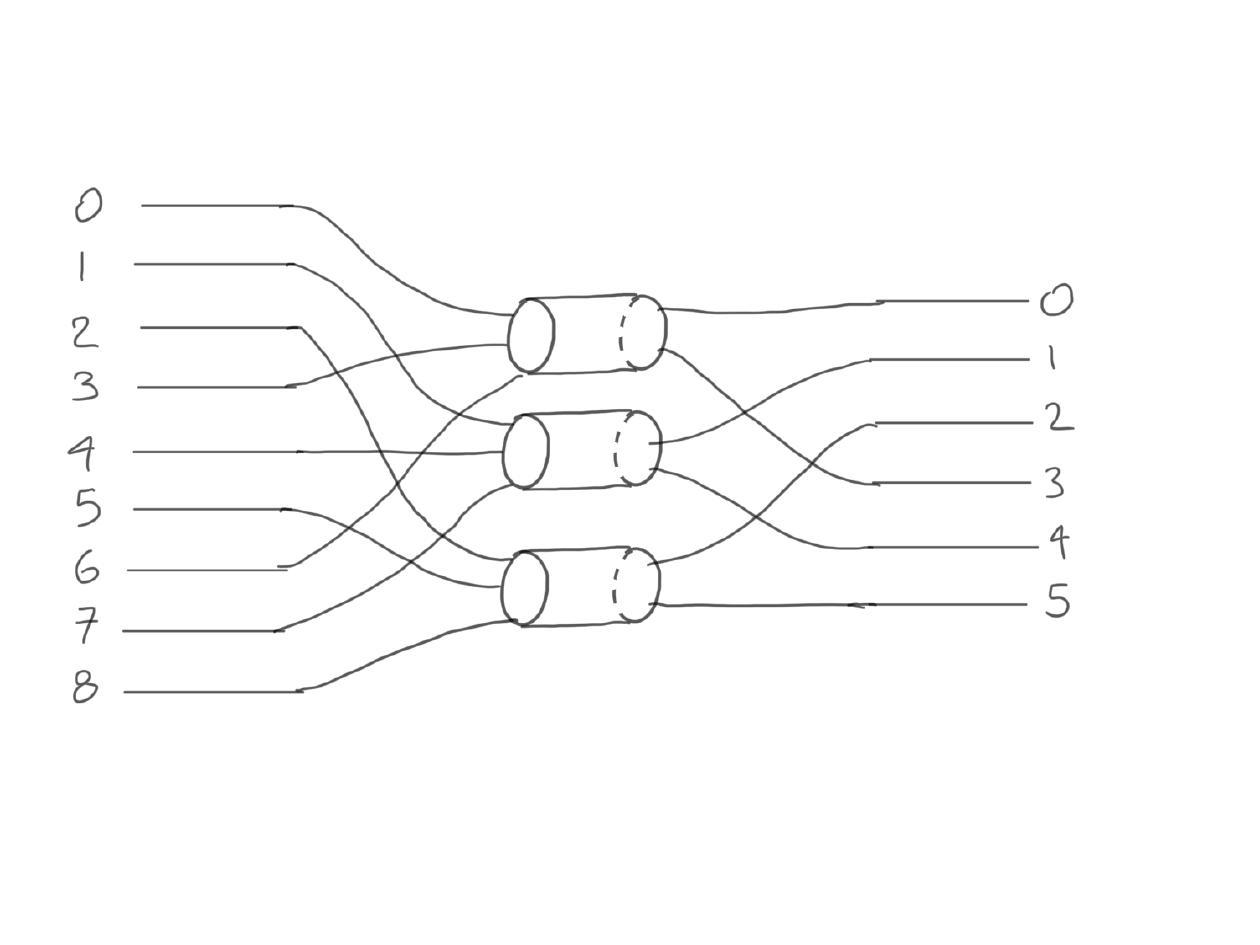}
\caption{A passage from one transversal section to another: $n=9, m=6$}. 
\label{16122021-2}
\end{center}
\end{figure}
In order to make the content of this section more accessible for applications we have made Figure \ref{17122021}, and an example of it in Figure \ref{16122021-2}, which shows the deformation retract of $L_t$ for which one can describe the action of monodromy. The points in $\Sigma_i\cap L_t,\ \ i=1,2$ are ordered according to the usual order of roots of unity and we identify them with  $\Sigma_1:=\{0,1,2,\ldots, n-1\}$ and $\Sigma_2:=\{0,1,2,\ldots,m-1\}$, respectively.  In this way 
\begin{eqnarray*}
 \Sigma_1\cap L_{t,h} &=&\{h,e+h,2e+h,\ldots,(q-1)e+h\},\\
 \Sigma_2\cap L_{t,h} &=& \{h,e+h,2e+h,\ldots,(p-1)e+h\}.
\end{eqnarray*}
In $\Sigma_1\cap L_{t,h}$ and $\Sigma_2\cap L_{t,h}$ we take minus $h$ and divide by $e$ and connect them to  $\Sigma_{1,h}:=\{0,1,2,\ldots, q-1\}$ and $\Sigma_{2,h}:=\{0,1,\ldots,p-1\}$, respectively.  We consider another copy $\Sigma_{1,h}'$ of $\Sigma_{1,h}$ and connect $x\in\Sigma_{1,h}$ to $x(-p)^{-1}\in \Sigma_{1,h}'$ modulo $q$ and another  copy $\Sigma_{2,h}'$ of $\Sigma_{2,h}$ connecting $x\in \Sigma_{2,h}$ to $x q^{-1}\in \Sigma_{2,h}'$ modulo $p$. Now, all the points of $\Sigma_{i,h}',\ i=1,2$ are connected to a single point $p_h$ for which we also consider a loop $\delta_h$ at $p_h$ with orientation. We now describe the monodromy. Consider a path $\gamma$ from $ie+h\in \Sigma_1$ to $je+h$ which turns in $\delta_h$, $s_\gamma\in \Z$ times. If $h<e-1$ the monodromy $h(\gamma)$ of $\gamma$ is a similar path starting from $ie+h+1$ and $je+h+1$ and turning in the loop $\delta_{h+1}$, $s_{\gamma}$ times. If $h=e-1$ then $h(\gamma)$ starts from $ie$ and ends in $jb$. If $\gamma$ passes through $k\in \Sigma_{1,e-1}'$ and $l\in \Sigma_{2,e-1}'$  then $h(\gamma)$ passes through  $k-a\in \Sigma_{1,0}'$ and $l-b\in \Sigma_{2,0}'$. The  number of turns in $\delta_0$ of $h(\gamma)$ is $s_\gamma+\left[\frac{k-a}{q}\right]+\left[\frac{l-b}{p}\right]$. 
\begin{figure}
\begin{center}
\includegraphics[width=0.6\textwidth]{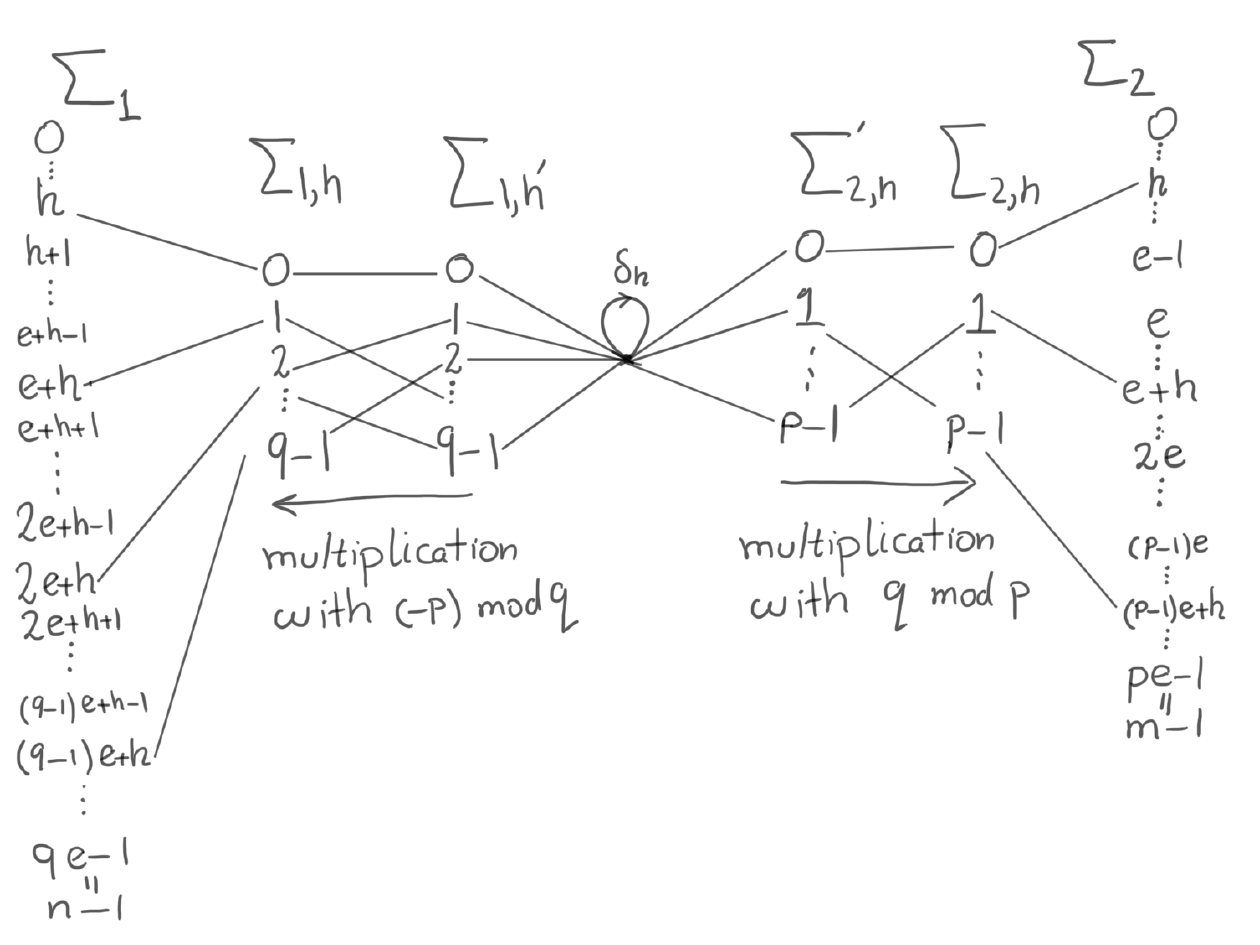}
\caption{The passage from $\Sigma_1$ to $\Sigma_2$.} 
\label{17122021}
\end{center}
\end{figure}

\section{Product of $d+1$ lines in general position}
\label{14dec2021}
We consider the polynomial 
\begin{equation}
\label{17nov2021}
f= l_1^{n_1} l_2^{n_2} \dots l_{d+1}^{n_{d+1}}. 
\end{equation} 
where $l_i$  are lines in a general position, and $n_i$ are positive integers without common divisors. 
We do not suppose that $n_i, n_j$ are relatively prime. 
Let 
$$
L_t = \{(x,y)\in \C : f(x,y) =t\}.
$$
\begin{theorem}
\label{1dec2021-1}
If  $t$ is a regular value of $f$, then
\begin{align}
\label{f7}
H_1(L_t, \Z)\cong \Z^{(d-1)(n_1+n_2+\cdots+n_{d+1})+1}.
\end{align}
\end{theorem}
\begin{proof}
We fix a fiber $X:=L_t$ with $t$ near to zero, consider the projection in $x$ coordinate  $\pi: X\to\C, \ (x,y) \to x$ and assume that the  parallel lines $x=$constant are transversal   to lines $l_i$ and any two intersection points of $l_i$'s have not the same $x$-coordinate.
It turns out that the set of critical points of $\pi$ is a union of  $\frac{d(d+1)}{2}$ sets $P_{ij}$ which is  near to $l_i\cap l_j$. Let $C_{ij}=\pi(P_{ij})$ and consider a regular point $b\in\C$ for $\pi$ and  $\Sigma:=\pi^{-1}(b)$. This  is a union of  $\sum_{i=1}^{d+1} n_i$ distinct points. Let also $D_{ij}$ be a small disc around $C_{ij}$ and $b_{ij}$ be a point in its boundary. A classical argument in the topology of algebraic varieties involving deformation retracts and excision theorem, see for instance \cite[5.4.1]{lam81},\cite[Section 6.7]{ho13},  gives us:
\begin{align}
\label{directsum}
H_1(X, \Sigma,  \Z) = \oplus_{ij} H_1( \pi^{-1}(D_{ij}), \pi^{-1}(b_{ij}),\Z) .
\end{align}
 Now $\pi^{-1}(D_{ij})$ is a union of ${}(n_i,n_j)$ cylinders with $n_i+n_j$ points 
from $\pi^{-1}(b_{ij})$ in its boundary (as in Section \ref{8dec2021}) and $(\sum_{k=1}^{d+1} n_k)-n_i-n_j$ discs, each one  with one point from $\pi^{-1}(b_{ij})$ in its boundary. Using Proposition \ref{10nov2021} we conclude that 
$$
H_1(X, \Sigma)\cong \Z^{d(n_1+n_2+\cdots+n_{d+1})} .
$$
The  long exact sequence of the pair $X,\Sigma$ finishes the proof.   
$$
\begin{array}{cccccccccc}
0 &\to & H_1(L_t)&\to& H_1(L_t, \Sigma)&\to& H_0(\Sigma )\to H_0(L_t)&\to& 0& \\
&& && \parallel
&&
\parallel
\ \ \ \ \ \ \ \ \ \ \ \ \ \ 
\parallel
&

&&
\\
&& && \Z^{d(n_1+n_2+\cdots+n_{d+1})}
&&
\Z^{n_1+n_2+\cdots+n_{d+1}}
\Z
&&&
\\
\end{array}
$$
\end{proof}

\begin{corollary}
\label{1dec2021-2}
The genus of the curve $L_t$ equals
$$
\frac{1}{2}\left(
(d-1)n+2-\sum_{i=1}^{d+1}{}(n_i, n)
\right),\ \ \ n:=\sum_{j=1}^{d+1}n_j,
$$
where $f:=l_1^{n_1}l_2^{n_2}\cdots l_{d+1}^{n_{d+1}}$ and 
$t$ is a regular value of $f$.
\end{corollary}
\begin{proof}
By genus of  $L_t$ we mean the genus  of the compactified and desingularized curve. The hypothesis ${\rm gcd}(n_1,n_2,\ldots,n_{d+1})=1$ and $t$ is not a critical value of $f$, together imply that the curve $f=t$ is an irreducible polynomial. The curve $f=t$ intersects the line at infinity $\P ^1$ at the intersection $p_i$ of the lines $l_i=0$ with $\P ^1$. Near $p_i$ our curve has $(n_i, n)$ local irreducible components because 
$$
f-t=\frac{(a_ix+by_i+c_iz)^{n_i}}{z^n}g_i
$$
where  $l_i=a_ix+by_i+c_i$ and  $g_i$ is a holomorphic function near $p_i$ with $g(p_i)\not=0$.  
\end{proof}
The genus of the degree three curve $\{ xy(x+y-1)=1 \}$ is one. The genus of  the degree six curve $\{ xy^2(x+y-1)^3=1 \}$ is also one. 
\begin{figure}
\begin{center}
\includegraphics[width=0.5\textwidth]{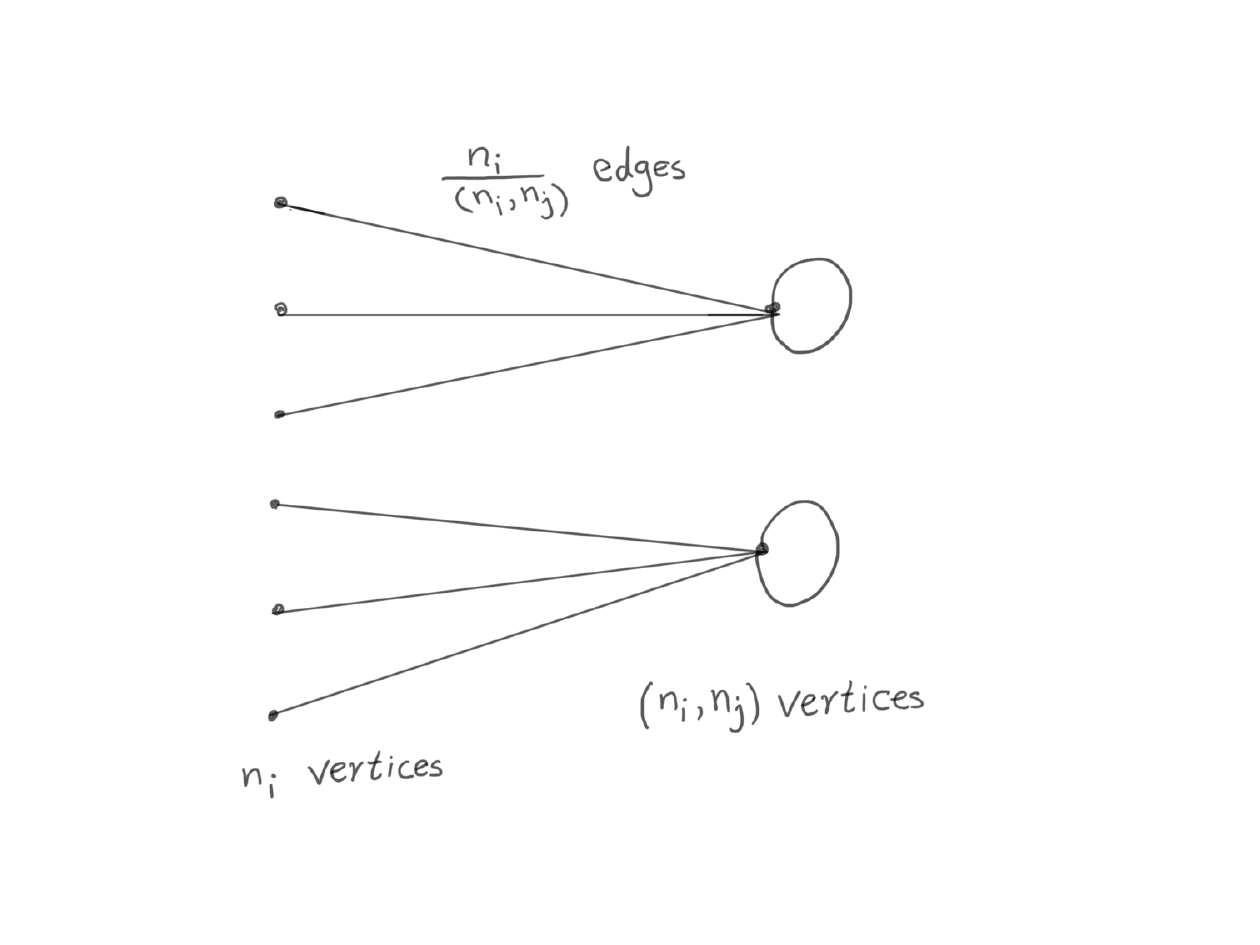}
\caption{Deformation retract of $L_t$}. 
\label{11122021-1}
\end{center}
\end{figure}

We are going to define a graph with $\sum_{i\not= j}(n_i,n_j)+\sum_{i=1}^{d+1}n_i$ vertices. 
The $(n_i,n_j)$ vertices   corresponds to the intersection points $l_i\cap l_j$. The $n_i$ vertices corresponds to the intersection of the line $\Sigma$ in the proof of Theorem \ref{1dec2021-1} with the fiber $L_t$, $t$ near to zero. 
Each group of $(n_i,n_j)$ vertices  are connected  with $n_i$ edges, each one with $\frac{n_{i}}{(n_i,n_j)}$ edges,  to $n_i$  vertices in the second group corresponding the intersection of $\Sigma$ with $L_t$ near $l_i=0$. This description is trivially  unique for $(n_i.n_j)=1$. If $(n_i,n_j)\not=1$ we have to determine the decomposition of $n_i$ vertices into $\frac{n_{i}}{(n_i,n_j)}$ sets of cardinality $(n_i,n_j)$. This might be done using the description of the deformation retract at the end of Section \ref{8dec2021}.    Moreover, we consider a loop for  each $(n_i,n_j)$ vertices. This will correspond to the saddle vanishing cycles. From the proof of Theorem \ref{1dec2021-1} it follows that
\begin{proposition}
\label{11dec2021}
The graph $G$ is a deformation retract of $L_t$. 
\end{proposition}
\begin{figure}
\begin{center}
\includegraphics[width=0.5\textwidth]{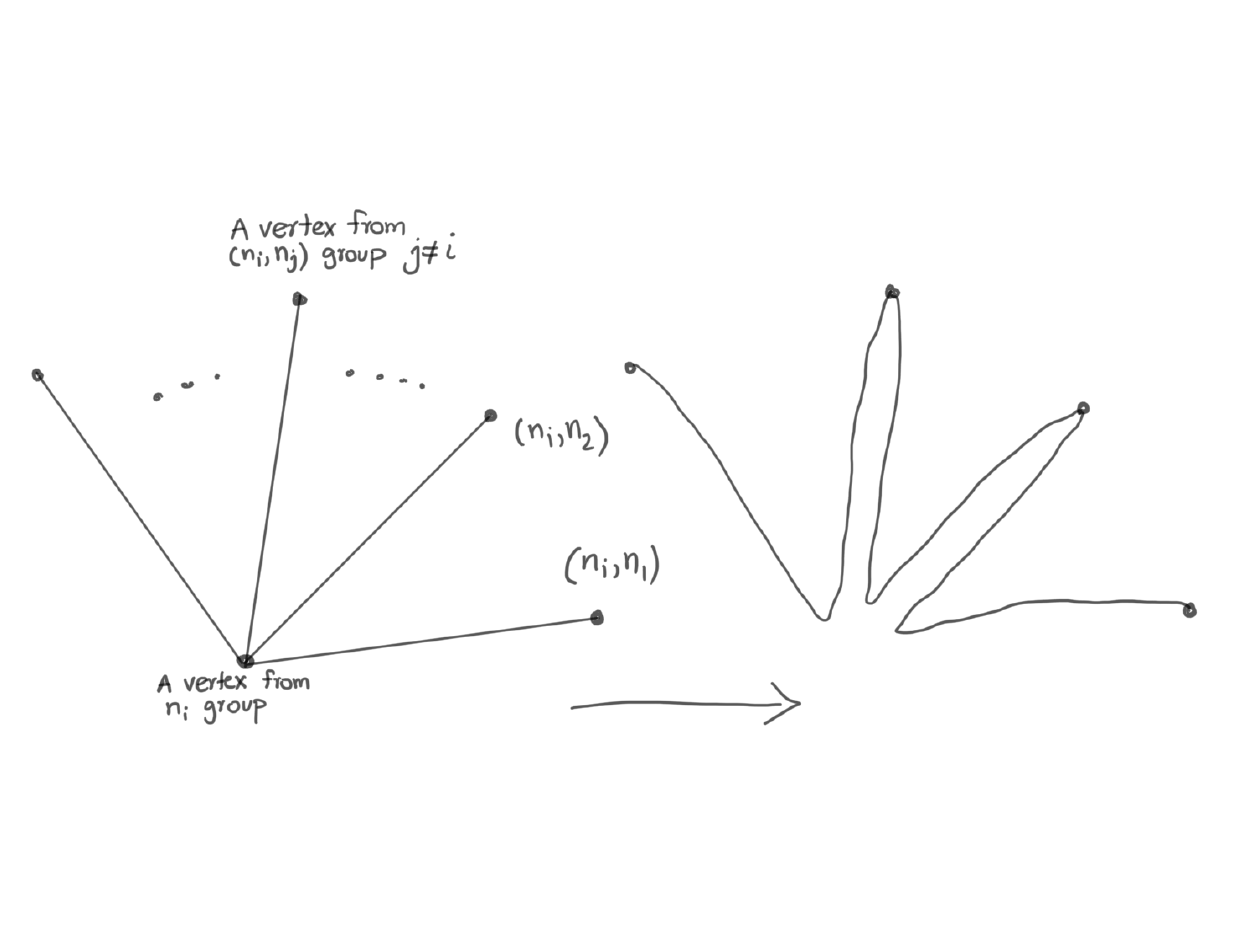}
\caption{Getting the graph $\check G$ from $G$}. 
\label{11122021-3}
\end{center}
\end{figure}
\begin{figure}[t]
\begin{center}
\includegraphics[width=0.5\textwidth]{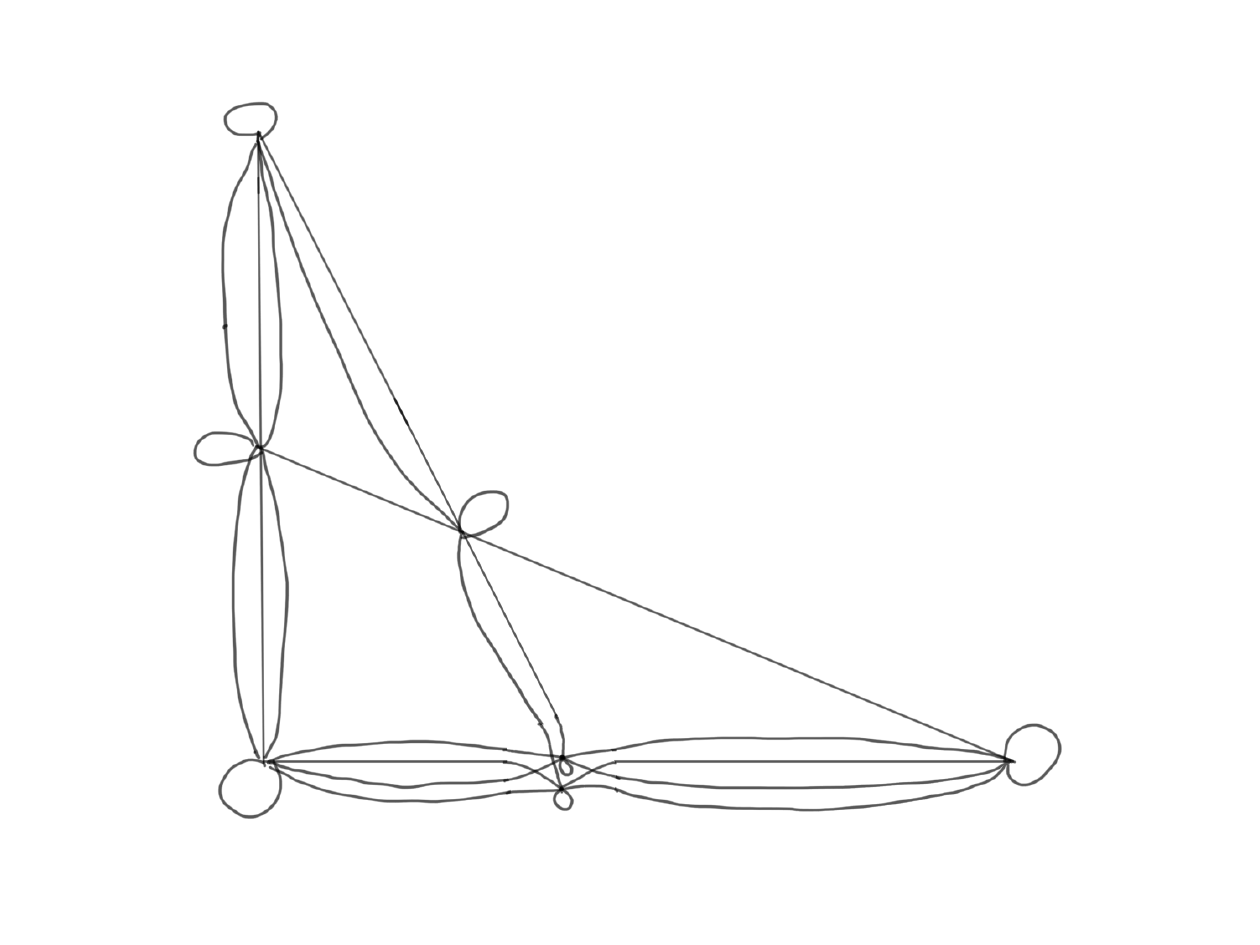}
\caption{Deformation retract of $L_t$ for four lines with multiplicities 1,2,3,4}. 
\label{11122021-2}
\end{center}
\end{figure}
For $f$ and lines $l_i$ defined over real numbers, there is another way to describe  a simpler graph $\check G$ which  shows the homotopy type of $L_t$.  Each vertex in the $n_i$ group is connected with $d$ edges to $d$ vertices corresponding to the intersection of $l_i$ with other lines. We order  them as they meet $l_i$.  We replace this with the one in Figure \ref{11122021-3} and we get a graph $\check G$ with $\sum_{i\not= j}(n_i,n_j)$ vertices which can be described easily using the real geometry of lines as follows.     
We cut out infinite segments of the union of lines $\cup_{i=1}^{d+1}l_i\subset\R^2$, replace each intersection point $l_i\cap l_j$ with $(n_i,n_j)$ vertices and  replace  each finite segment which connects $l_i\cap l_j$ to $l_i\cap l_k$ (and does not intersects other lines in its interior)  with $n_i$ edges connecting $(n_i,n_j)$ vertices with $(n_i,n_k)$ vertices, provided that each vertex in the first and second group  has only $\frac{n_i}{(n_i,n_j)}$ and $\frac{n_i}{(n_i,n_k)}$ edges respectively. Moreover, we consider a loop in each $(n_i,n_j)$ vertices. We obtain the new graph $\check G$. 
   \begin{remark}\rm
The the deformation retracts above appeared first in the study of the topology and the monodromy of the logarithmic foliation with first integral
$f=x^py^p(1+x+y)$ in \cite{boga11} in relation with the classical paper \cite{giho85}.
\end{remark}

\section{Computation of intersection indices}
\label{02dec2021}
The computation of intersection indices between vanishing cycles is an important ingredient in the study of deformations of singularities. By analogy we define intersection index for paths  in the leafs of a holomorphic  foliation. Our main result Theorem \ref{02dec2021} in this section is a  generalization of a theorem by S. Gusein-Zade and N. A'Campo, see \cite[Section 2]{mov}.

Let us consider a holmorphic foliation $\F(\omega)$  in $\R^2$ given by a polynomial 1-form $\omega$ with real coefficients. We consider  an open subset $U\subset \R^2$ with exactly two saddle singularities $O_1$ and $O_2$ of $\F$ and assume that $O_1$ and $O_2$ have a common separatrix. We assume that the 1-forms $\omega$ near $O_i,\ i=1,2$ in local coordinates $(x_i,y_i)$ is given by the linear equation    $x_idy_i+\alpha_i y_i dx_i,\ \alpha_i >0$, and so, it has the meromorphic local first integral $y_ix_i^{\alpha_i}$. In a neighborhood of $O_i$ the foliation has two separatrices $x_i=0$ and $y_i=0$.
The common separatrix   is given by $y_i=0$. We consider transversal sections to $\F$ at the points $b_0,b_1,b_2$ respectively in the common separatrix, $x_1=0$ and $x_2=0$. Let $\gamma_0$ be the real trajectory of $\F$ which connects a point $p_1\in \Sigma_1$ to $p_2\in\Sigma_2$ crossing the point $p_0\in\Sigma_0$, see Figure \ref{10dec2021-1}.

We consider now the complex foliation $\F$ in $\C^2$ and use the same notation for complexified objects.   We consider  a  path $\lambda:\R\to \Sigma_0$ which has period one and  restricted to $[0,1]$ turns  once around $b_0$ anticlockwise.  The path $\gamma_0$  from $p_1$ to $p_2$ can be lifted  to a unique path  $\gamma_t$ in a leaf of $\F$  which crosses  $\lambda(t)\in\Sigma_0$  and connects $q_1(t)\in\Sigma_1$ to $q_2(t)\in\Sigma_2$. This lifting in general is not possible, however,  in our situation this follows from the fact that  $O_1$ and $O_2$ are linearizable and $\alpha_i\in\R$. Since $\alpha_i>0$,  the trace of $q_i(t)$ in $\Sigma_i$  will give us paths $\lambda_i$ in $\Sigma_i$ turning around $b_i$ anticlockwise. 
If we assume that $\lambda(\frac{1}{2})$ is again in the real domain $\R^2$ and it lies in a real leaf $\gamma$ of $\F$ in the other side of the common separatrix,  then  we have the main result of this section. 
\begin{theorem}
\label{10dec2021}
With the notations as above
\begin{equation}
\label{01032022}
\langle \gamma_{\frac{1}{2}},\gamma\rangle=+1,  
\end{equation}
where we have oriented $\gamma$ from $O_1$ to $O_2$. 
\end{theorem}
\begin{proof}
  We look at the projection $\pi(\gamma_t)$ of the path $\gamma_t$ in the common separatrix and we see Figure \ref{10dec2021-2}. The projection can be constructed in a $C^\infty$ context by gluing the local transversal foliations $dx_i=0,\ i=1,2$. The intersection number is not changed under this projection and \eqref{01032022} follows. 
\end{proof}
\begin{figure}[h]
\begin{center}
\includegraphics[width=0.7\textwidth]{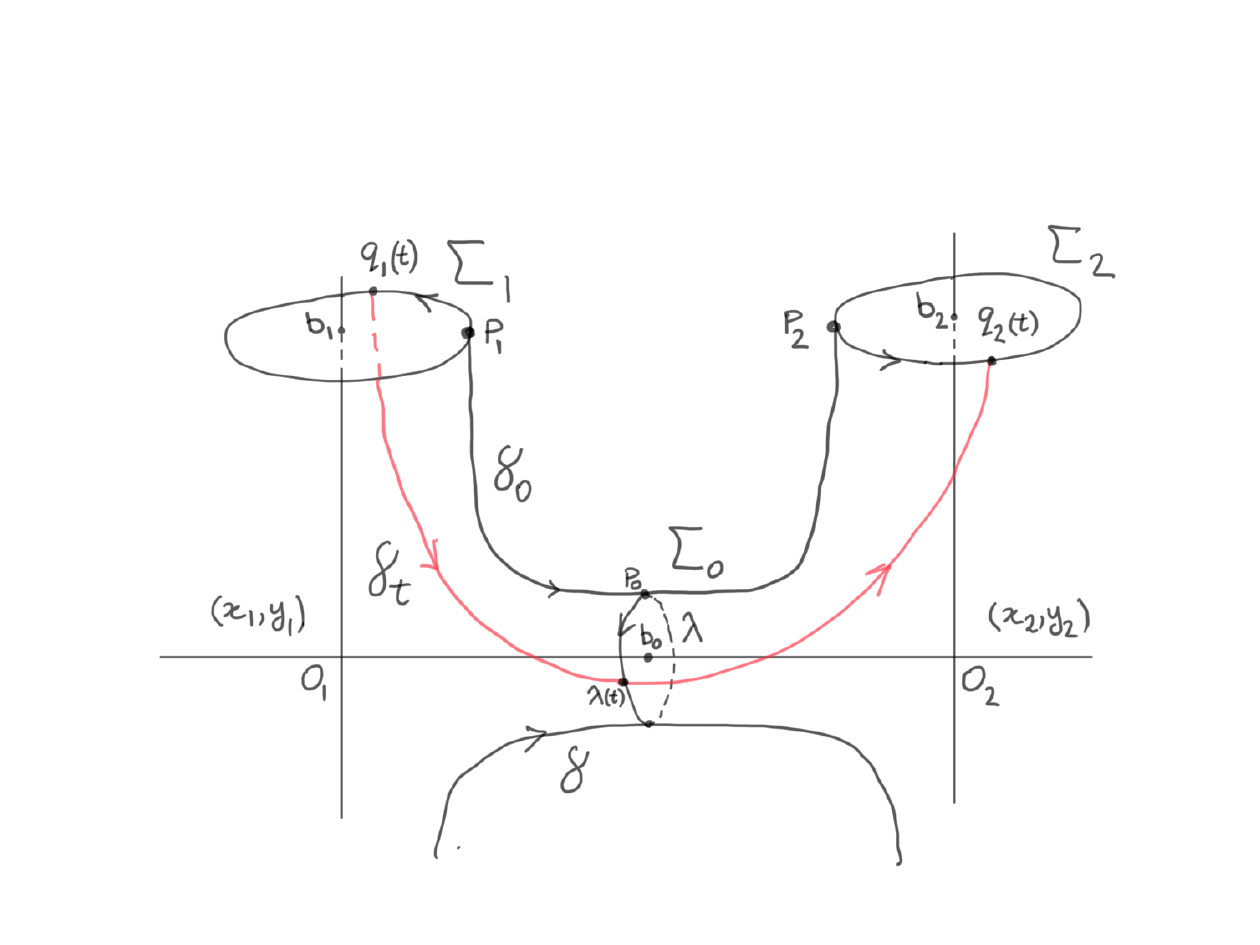}
\caption{Two saddles}. 
\label{10dec2021-1}
\end{center}
\end{figure}
\begin{figure}[h]
\begin{center}
\includegraphics[width=0.5\textwidth]{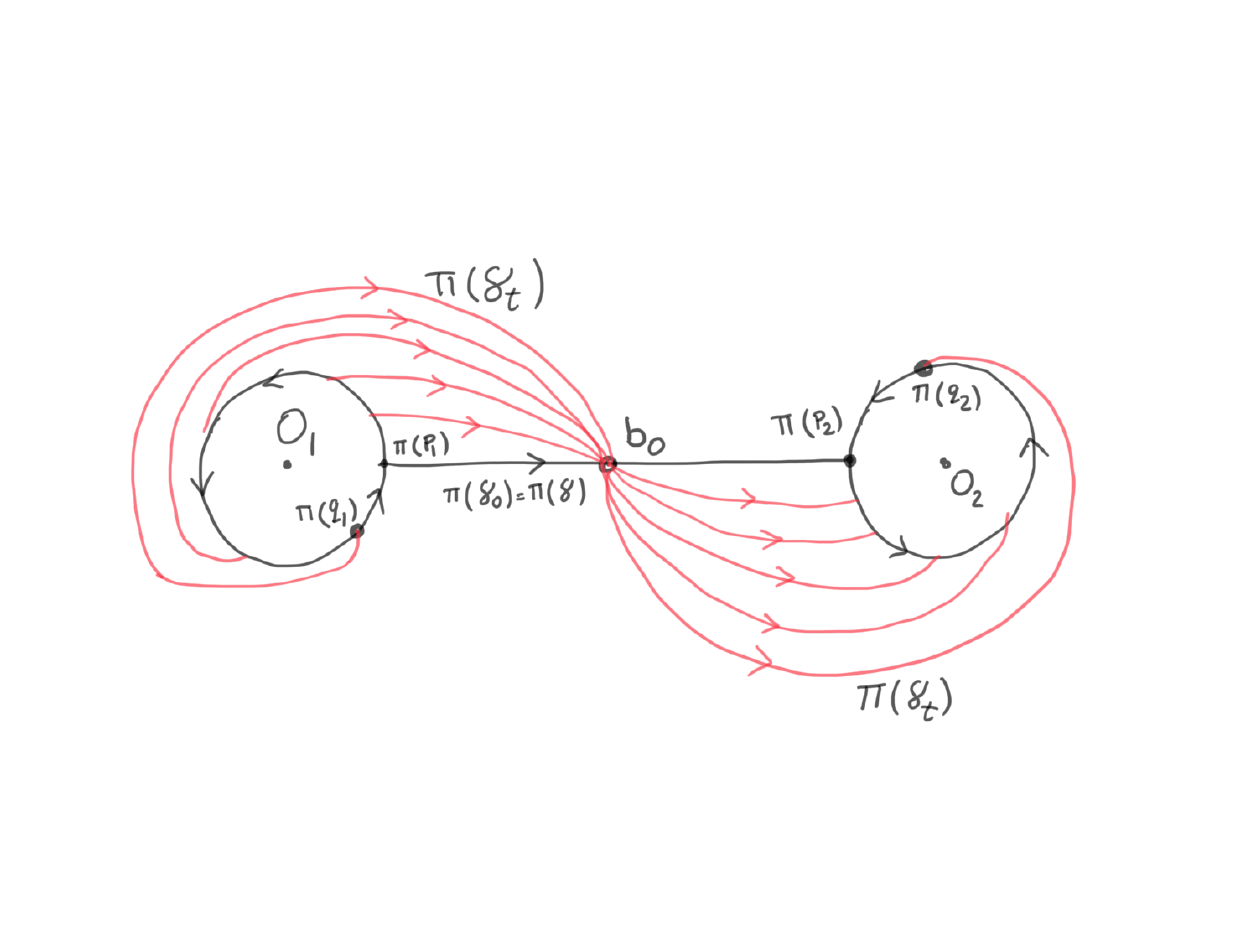}
\caption{Projection into the common separatrix}. 
\label{10dec2021-2}
\end{center}
\end{figure}

\begin{remark}\rm
The above proof uses arguments close to the one  used by  Gusein-Zade \cite{guse74} for germs of isolated singularities. 
A second proof can be produced by making use of more elaborated arguments of A'Campo \cite[p.23-24]{acam75} as follows. Consider the complex-conjugate path 
$\overline{\gamma_{\frac12}}= \gamma_{-\frac12}$. As the complex conjugation inverses the orientation then 
$\langle \gamma_{-\frac{1}{2}},\gamma\rangle=-\langle \gamma_{\frac{1}{2}},\gamma\rangle$. On the other hand the class $ \gamma_{\frac{1}{2}} - \gamma_{-\frac{1}{2}}$ can be represented by 
two disjoint paths $\alpha_1, \alpha_2$ connecting $\Sigma_1$ to $\Sigma_1 $ and  $\Sigma_2 $ to $\Sigma_2 $ respectively. These paths define geometrically the holonomy of the two vertical separatrices. It remains to compute the intersection index of $\alpha_1$ (representing holonomy) and the 
class of $\gamma_{\frac{1}{2}}$ (representing the Dulac map near $O_1$). This is of course a local computation in a neighborhood of $O_1$ and it follows from the local description of a complex saddle that their intersection index equals one. Similar computation holds for $\alpha_2$ from which the result follows.

A third proof can be obtained by deformation. Namely, it suffices to note that the intersection index depends continuously on parameters, hence it is a constant. Such a deformation is possible in any compact interval for the parameter, provided that the initial and end points of the path on the cross-section $\Sigma_1$ and $\Sigma_2$ are  sufficiently close to the vertical separatrix. Therefore, it is enough to check the claim of the theorem for some toy example, like
$df=0$ with $f=(x^2-1)y$, in which an explicit computation of different paths and their deformations is possible. 
\end{remark}
\begin{remark}\rm
Theorem \ref{10dec2021} holds true without the assumption that the saddle are linearizable (with similar proof). We only need to know the asymptotic behavior of the Dulac map.
\end{remark}
\section{The orbit of a vanishing cycle of center type}
\label{14.12.2021}
Let us consider the polynomial $f$ in $\C^2$ given by \eqref{17nov2021}. 
 The map $f : \C^2 \to \C$
defines a locally trivial fibration over the $\C\backslash C$, where $C$ consists of $\frac{d(d-1)}{2}+1$ critical values of $f$. It is the union of the values of critical points of center type (which we assume that such critical values are distinct) and  the critical value $0$ over the  
$\frac{d(d+1)}{2} $ saddles points which are intersection of lines.  We  choose a point  $b\in\C$ with ${\rm Im}(b)>0$ and fix straight paths $\gamma_c,\ c\in C$ joining $b$  to the critical values of $f$ (a distinguished set of paths). Let also $h_{c}: H_1(L_b,\Z)\to H_1(L_b,\Z)$ be the monodromy along $\gamma_c$ until getting near to $c$, turning around $c$ anticlockwise and returning to $b$ along $\gamma_c^{-1}$.  
Let $\delta_c\in H_1(L_b,\Z)$ be the center vanishing cycle along $\lambda_c,\ c\not=0$. Along $\gamma_0$ we get $\sum_{i\not =j} {}(n_i,n_j)$ saddle vanishing cycles in $H_1(L_b,\Z)$. 
 We denote by $\overline{L_b}$ the curve obtained by a smooth compactification  of $L_b$. 
\begin{theorem}
\label{camacho2021}
Assume that $n_i$'s are pairwise coprime. The  $\Q$-vector space $O_{\delta}\subset H_1(L_b, \Q)$ generated by the action of monodromy on a fixed center vanishing cycle $\delta$ has  codimension $d$ in $ H_1(L_b, \Q)$. Moreover, 
\begin{equation}
\label{18122021france}
O_{\delta}=\left \{\gamma \in H_1(L_b,\Q) \Big|\int_{\gamma} \frac{dl_i}{l_i} =0, i= 1,\dots, d+1 \right \}
\end{equation}
and the restriction of the map 
$ H_1(L_b, \Q)\to  H_1( \overline{L_b},\Q)$ induced by inclusion, to $O_\delta$ is surjective. 
\end{theorem}
\begin{proof}
 Let $S\subset H_1(L_b,\Q)$ be the $\Q$-vector space generated by saddle vanishing cycles. We first compute the action of monodromy in $H_1(L_b,\Q)/S$. For this we prove that all center vanishing cycles are in $O_\delta$.  Consider center vanishing cycles  $\delta=\delta_{c_1}$, $\delta_{c_2}$,  the critical points $p_1,p_2$ with $f(p_i)=c_i$  which are inside two adjacent polygons  $P_1$ and $P_2$ formed by the lines $l_i=0$. Let $l_1$ be the line of the common  edge which has multiplicity $n_1$. We are in the situation of Theorem \ref{10dec2021}. The restriction of the map $f$ to $\Sigma_0$ in a local coordinate $z$ in $\Sigma_0$ is given by $z\mapsto z^{n_1}$. Let $p_0,\tilde p_0$ be two points in the real transversal section $\Sigma_0$ in $\R^2$, $p_0$ above and $\tilde p_0$ under the line $l_1$, and $\lambda$ as in before Theorem \ref{10dec2021}. The image of $\lambda$ under $f|_{\Sigma_0}$ is a path which starts at $f(p_0)$ and turns $\frac{n_1}{2}$ times around $0\in\C$. The conclusion is that $h_0^{[\frac{n_1}{2}]+\epsilon}(\delta_{c_1})$ has a non-zero intersection with $\delta_{c_2}$, where $\epsilon=0$ if $f(p_0)>0$ and $\epsilon=1$ if $f(p_0)<0$. Using the classical Picard-Lefschetz formula, see Theorem \ref{12nov2021}, we conclude that $\delta_{c_2}\in O_{\delta}$. Further applications of Picard-Lefschetz formula  will imply that all center vanishing cycles are in $O_\delta$.
 
Our hypothesis on $n_i$'s implies that over the point $l_i\cap l_j$ we have exactly one saddle vanishing cycle.
For a finite polygon in the complement of $\cup_{i=1}^{d+1}l_i$ in $\R^2$, let $a_1,a_2,\ldots,a_s$ be the multiplicity of its edges formed by the lines $l_{1},l_2,\ldots, l_s$. Let also $\delta$ be the center vanishing cycle inside this polygon.  We look $\delta$ in the deformation retract $G$ of $L_b$ in Proposition \ref{11dec2021}. The monodromy $h^{a_1a_2\cdots \hat{a_i}\cdots a_s}(\delta)$, where $\hat{a_i}$ means $a_i$ is removed, fixes all the edges of $\delta_i$ except for the $i$-the edge, and its iteration will replace its $i$-th edge with any other $a_i$ paths in the deformation retract of $L_b$. Moreover, any path in $H_1(G,\Z)/S$ is a linear combination of center vanishing cycles. 
The conclusion is that the action of monodromy on $H_1(L_b,\Q)/S$ generates the whole space. By the classical Picard-Lefschetz formula we have
\begin{equation}
\label{12dec2021}
h^{a}(\delta)=\delta+\sum_{i=1}^s\frac{a}{a_ia_{i+1}}\delta_{i,i+1},\ \ \ a:=a_1a_2\cdots a_s,\ \ \  s+1:=1,
\end{equation}
where $\delta_{i,i+1}$ is the saddle vanishing cycle over $l_i\cap l_{i+1}$. It follows that by the action of monodromy  we can generate a sum  of saddle vanishing cycles as above. It is easy to see that these elements are linearly independent in $H_1(L_b,\Q)$.  The codimesnion in $S$ of the $\Q$-vector space generated by these elements is exactly $d$.  

Let $\delta_{i,h},\ i=1,2,\ldots,d+1,\ h=1,2,\ldots, (n,n_i)$ be the closed cycles around the points at infinity $p_{i,h}$ of $L_b$ corresponding to the intersection of $l_i$ with the line at infinity. An easy residue calculation shows that 
\begin{equation}
\label{18122021toulouse}
\int_{\delta_{j,h}}\frac{dl_i}{l_i}=\left\{
\begin{array}{cc}
\frac{n-n_j}{(n,n_j)}     &  i=j \\
 \frac{-n_j}{(n,n_j)}     &   i\not= j
\end{array}.
\right.
\end{equation}
This shows that cohomology classes of the $d+1$ logarithmic one-forms $\frac{dl_i}{l_i}$ in $H^1_{dR}(L_b)$
generate a vector space of dimension $d$ (there is one linear relation between these forms restricted on $L_b$). The equality \eqref{18122021france} follows, as both sides of the equality  are of codimension $d$ and $\subseteq$ is trivially true. Moreover, by \eqref{18122021toulouse} we have $\delta_{i,h}-\delta_{i,0}\in O_{\delta},\ \ i=1,2,\ldots,(n,n_i)-1$ and $H_1(L_b,\Q)$ is a direct sum of $O_\delta$ with the the $\Q$-vector space generated by $\delta_{i,0},\ i=1,2,\ldots,d$.    
\end{proof}
\begin{remark}\rm
A purely topological argument for the last part of the proof of Theorem \ref{camacho2021} can be formulated following \cite[Section 2]{mov} and it is as follows. 
Let us orient the center vanishing cycles using the anticlockwise orientation of $\R^2$. We can orient the saddle vanishing cycle $\delta$ attached to $l_i\cap l_j$ in such a way that  it intersects positively the center vanishing cycles in the finite polygons with $l_i\cap l_j$ vertex. For any line $l_i$, let $\delta^i$ be the alternative sum of saddle vanishing cycles in the order which $l_i$ intersects others. It turns out that  the intersection of $\delta^i$'s with center vanishing cycles is zero.  Since it is invariant under monodromy $h_0$ around $0$, its intersection with all $h^{k}(\delta)$, $\delta$ center vanishing cycle, is also zero. The conclusion is that the intersection of $\delta^i$ with all the elements in $H_1(L_b,\Z)$ is zero and hence it is in the kernel of  $H_1(L_b, \Q)\to  H_1( \overline{L_b},\Q)$. After taking a proper sign for $\delta^i$, we know that $\sum_{i=1}^{d+1}\delta^i=0$. Now, it is an elementary problem to check that $\delta^i,\ i=1,2,\ldots,d$ and $\frac{d(d-1)}{2}$ elements \eqref{12dec2021} attached to each polygon are linearly independent and form a basis for the vector space generated by saddle vanishing cycles.   
\end{remark}

\section{Proof of Theorem \ref{main1}}
\label{14/12/2021}
Let $\C[x,y]_{\leq 1}$ be the complex vector space of bi-variate complex polynomials of degree at most one.
The set ${\mathcal L}(1^{d+1})$ of logarithmic foliations (\ref{omicron2021}), is parameterized by the map
\begin{equation}
\label{09052022ivan}
\tau:\C^{d+1}\times \C[x,y]_{\leq 1}^{d+1}\rightarrow\F(d)
\end{equation}
\begin{equation}
\label{29apr02}
\tau(\lambda_1,\ldots,\lambda_{d+1},l_1,\ldots,l_{d+1})= 
l_1\cdots l_{d+1}\sum_{i=1}^{d+1} \lambda_i\frac{dl_i}{l_i}
\end{equation}
and  hence it is an irreducible algebraic set. The  differential 
\begin{align}
\label{diff}
D \tau (n_1,\dots,n_{d+1},l_1,\ldots,l_{d+1})
\end{align}
of $\tau$  at the point 
\begin{align}
\label{omicron0}
\omega_0 = l_1 l_2\cdots l_{d+1} (\sum_{i=1}^{d+1} n_i\frac{dl_i}{l_i})\in {\mathcal L}(1^{d+1})
\end{align}
applied to the vector $ (\lambda_1,\ldots,\lambda_{d+1},p_1,\ldots,p_{d+1})$
is
\begin{equation}
    \label{1300toulouse}
l_1l_2\cdots l_{d+1}\left\{ \sum_{i=1}^{d+1} \lambda_i\frac{dl_i}{l_i} +  
\left(\sum_{i=1}^{d+1}\frac{p_i}{l_i}\right)\left(\sum_{i=1}^{d+1} n_i\frac{dl_i}{l_i}\right)+ d\left(\sum_{i=1}^{d+1}n_i\frac{p_i}{l_i}\right) 
\right\}.
\end{equation}
It is easy to check that if a vector $(\lambda_1,\ldots,\lambda_{d+1},p_1,\ldots,p_{d+1})$ is in the kernel of $D\tau$ then for every $i$ the polynomial $p_i$ is colinear to $l_i$. It follows from this that the dimension of the kernel is $d+1$, respectively the dimension of the image (\ref{1300toulouse}) is $3(d+1)$.

Let $\omega_0$ be the polynomial one-form defined by (\ref{omicron0}) which we assume to be with real coefficients. We denote by 
$\delta_t \subset \{f=t\}$ a continuous  family of real vanishing cycles around a real center of 
$\omega_0$, where the parameter $t$ is the restriction of the first integral $f= \Pi_{i=1}^{d+1} l_i^{n_i}$ to a cross-section to $\{f=t\}$.
Let 
\begin{equation}
\label{05.12.2021}
\F_\epsilon: \ \ \omega_\epsilon:=\omega_0+\epsilon\omega_1+\cdots  
\end{equation}
be an arbitrary degree $d$ deformation of $\F_0$. The return map associated to the family $\delta_t$ and the deformation $\F_\epsilon$ takes the form
$$
t + \varepsilon \int_{\delta_t} \tilde\omega_1 + O(\varepsilon), \;\; \tilde\omega_1=\frac{\omega_1}{l_1\cdots l_{d+1}}
$$
\begin{proposition}
\label{tangent}
The Melnikov integral $M_1(t):= \int_{\delta_t} \frac{\omega_1}{l_1\cdots l_{d+1}}$ vanishes identically if and only if $\omega_1$ 
 is of the form (\ref{1300toulouse}).
\end{proposition}
\begin{proof}
It is trivially seen that if $\omega_1$ belongs to the image of $D\tau$ (\ref{1300toulouse}), then $\int_{\delta_t} \tilde\omega_1$
vanishes identically. We shall prove now the converse.  If $\int_{\delta_t} \tilde\omega_1$ vanishes  then it vanishes on every other family of cycles which are in the orbit of $\delta_t$, and hence on the vector space $O_{\delta_{t}}$ spanned by the orbit.  By Theorem \ref{camacho2021} the dual of $O_\delta$ in $H_{dR}^1(L_t)$ has a basis given by $\frac{dl_i}{l_i},\ i=1,2,\ldots, d$. Therefore, there are unique constants $\lambda_i,i=1,2,\ldots,d$ (depending on $t$) such that 
the cohomology class of
\begin{align}
 \tilde \omega_1- \sum_{i=1}^{d} \lambda_i \frac{dl_i}{l_i}
\end{align}
in $H^1_{DR}(L_t)$ is zero. Using \eqref{18122021toulouse}, we can see that $\lambda_i$ is a linear combination of integrals $\frac{1}{2\pi i}\int_{\delta_{i,0}}\tilde\omega_1,\ \ i=1,2,\ldots,d$ with rational, and hence constant, coefficients. We assume that $n_1,n_2,\ldots, n_d>1$.  Multiplying
$\int_{\delta_{i,0}}\tilde\omega_1$ with  $t^{\frac{1}{n_i}}$ and $t^{\frac{-1}{n}}$ we can see that it grows like $t^{-\frac{1}{n_i}}$ as $t\to 0$ and like  $t^{\frac{1}{n_i}}$ as $t^{-1}\to 0$. Since $\lambda_i$'s are uni-valued, it follows that they are meromorphic function in $t$ with pole order $<1$ at $t=0,\infty$ and hence they are constant.  

With the same arguments as in  \cite[Theorem 4.1]{mov0} we deduce that if a one-form on $L_t$ is co-homologically zero, then it is relatively exact, that is to say
\begin{equation}
\label{5dec2021}
\tilde\omega_1-\sum_{i=1}^{d+1} \lambda_i\frac{dl_i}{l_i}=d\tilde P+\tilde Q\tilde \omega_0, 
\end{equation}
where $\tilde P$ and $\tilde Q$ have only poles of order $\leq 1 $  along the lines $l_i=0$ and the line at infinity and $\tilde\omega_0=\frac{\omega_0}{l_1l_2\cdots l_{d+1}}$. 
The crucial observation is that the one-form (\ref{5dec2021}) is logarithmic along the line at infinity (after compactifying $\C^2$ to $\P^2$). Namely,   $\omega_1$ is of (affine) degree $\leq d$ which implies   $\tilde \omega_1$ and $d\tilde\omega_1$ have a pole of order at most one along the infinite line of $\P^2$. This implies that $d\tilde Q\wedge \tilde\omega_0$ has pole order $\leq 1$ at infinity, and hence $\tilde Q$ is holomorphic at infinity, and  by the equality \eqref{5dec2021}, $\tilde P$ is also holomorphic at infinity. The conclusion is that we can write 
$$
\tilde P=\frac{P}{l_1l_2\cdots l_{d+1}}, \ \ \tilde Q:=\frac{Q}{l_1l_2\cdots l_{f+1}},
$$
where $P,Q\in\C[x,y]$ are polynomials of degree $\leq d+1$. Multiplying the equality \eqref{5dec2021} with $l_1l_2\cdots l_{d+1}$ and considering it modulo $l_i=0$, we get $l_i|P-Qn_i$. If $n_i\not=n_j$ this implies that $P$ and $Q$ vanishes in the intersection points $l_{i}\cap l_j$. Knowing the degree of $P$ and $Q$, we conclude that both $P$ and $Q$ are of the form $l_1l_2\cdots l_{d+1}(\sum_{i=1}^{d+1}a_i\frac{p_i}{l_i})$, where $\deg(p_i)\leq 1$ and $a_i\in\C$ depend on $P,Q$. Substituting this ansatz for $P$ and $Q$ in \eqref{5dec2021} we get the desired form of $\omega_1$ in \eqref{1300toulouse}.
\end{proof}

The geometric meaning of the above Proposition  is that the tangent space of ${\mathcal L}(1^{d+1})$ and $\M(d)$ at the point $\F_0$ are the same, and that they are given by (\ref{1300toulouse}). By dimension count the dimension of this tangent space equals the dimension of ${\mathcal L}(1^{d+1})$. Therefore $\F_0$ is a smooth point on
 $\M(d)$ and moreover ${\mathcal L}(1^{d+1})$ is an irreducible component of the center set $\M(d)$.
Note that there are no other components of $\M(d)$ containing $\F_0$ and tangent to ${\mathcal L}(1^{d+1})$, otherwise the Zariski tangent space would be bigger.
  \begin{remark}
  \label{rem6}
  \rm
  Assuming Corollary \ref{tangent} we may complete the proof of Theorem \ref{main1} by general arguments which we sketch in what follows, 
  see \cite{gavr20}.
 Let $\mathcal B = (b_1,b_2,\dots, b_N)$
 be generators of the Bautin ideal near the point $\F_0\in \F(d)$ defined in Theorem \ref{main1}. Here $\mathcal B$ is seen as an ideal of the local ring of convergent power series. Following \cite[section 3.2]{gavr20} we divide the return map $P(t)$ associated to the family of vanishing cycles $\delta_t$ of $\F_0$ under a general (multi-parameter) perturbation in $\F(d)$ of the form
 $\omega_0+ \omega_1$  to obtain
 $$
 P(t)-t= \sum b_i(\Phi_i(t) + \dots).
 $$
 Here the dots replace some convergent power series in the parameters and $t$, which vanish at $\F_0$.
 Every Melnikov function (of arbitrary order) is an appropriate linear combination of $\Phi_i(t)$, although the converse is not necessarily true. In particular we may suppose that the vector space of 
 first order (or linear) Melnikov function has a basis formed by $\Phi_i(t), i=1,\dots, k$. The linear independence of the first order Melnikov function is equivalent then to the linear independence of the differentials $D b_i , i=1,\dots, k$ computed at $\F_0$,  see \cite[Corollary 2]{gavr20}.
  Thus, performing a local bi-analytic change of variables in the parameter space, we may assume that $b_1, b_2, \dots, b_k$ are new parameters defining a plane which contains ${\mathcal L}(1^{d+1})$ near $\F_0$. By comparison of dimensions we conclude that ${\mathcal L}(1^{d+1})$  coincides to the plane $\{b_1=\dots=b_k=0\}$  near $\F_0$. As for the other generators $b_i, i>k$ they already belong to the ideal $(b_1,\dots,b_k)$, because they vanish identically along 
 ${\mathcal L}(1^{d+1})$. Thus $\M(d)$ coincides near $\F_0$ with  ${\mathcal L}(1^{d+1})$.
  \end{remark}
  
 \begin{remark}\rm
Our hypothesis in Theorem \ref{main1} suggests to study the subset of $(\Q^+)^{d}$ given by points $(\frac{n_1}{n_{d+1}}, \frac{n_2}{n_{d+1}},\ldots, \frac{n_d}{n_{d+1}})$, where $n_i$'s are pairwise relatively prime positive integers. For instance, it is not clear whether this set is dense in $(\Q^+)^{d}$ or not. Note that its projection in each coordinate is dense and the fibers of this projection are finite sets. 
\end{remark}

\section{Quadratic foliations}

For quadratic foliations, that is the case $d=2$, the classification of components of  $\M(2)$ follows from the computations of  H. Dulac in \cite{Dulac1923}, see \cite[Appendix A]{gavr20}, \cite[Theorem 1.1]{LinsNeto2014} and \cite[Section 13.9]{IlyashenkoYakovenko}.   The algebraic set $\M(2)$ has four components 
\begin{enumerate}
\item
${\mathcal L}(1^3)$, 
\item
 the set ${\mathcal L}(1,2)$ of logarithmic foliations of the form 
 $$f_1f_2(\lambda_1\frac{df_1}{f_1}+\lambda_2\frac{df_2}{f_2}),\ \deg(f_1)=1,\deg(f_2)=2, \lambda_1,\lambda_2\in\C-\{0\}$$ 
\item
 the set ${\mathcal L}(3)$ of Hamiltonian foliations $\F(df),\ \deg(f)=3$ 
\item  an exceptional  component obtained by the action of yhe affine group $\Aff(\C^2)$ on the foliation with the first integral $\frac{(x^2+2y+\alpha)^3}{(x^3+3xy+1)^3}, \ \alpha\in\P^1$
\end{enumerate}
see  \cite[Proposition 4.7]{Gavrilov2020}. 
Using this, one may prove the following:
the only singular points of  $\M(2)$ in ${\mathcal L}(1^3)$ are ${\mathcal L}(1^3)\cap {\mathcal L}(1,2)$ and ${\mathcal L}(1^3)\cap {\mathcal L}(3)$, that is,
\begin{equation}
\label{09052022toulouse}
{\rm Sing} \;\M(2)\cap {\mathcal L}(1^3)=\left({\mathcal L}(1,2)\cap {\mathcal L}(1^3)\right)\cup\left(  {\mathcal L}(3)\cap {\mathcal L}(1^3)\right).
\end{equation}
A finer result is the classification of the components of the Bautin scheme which is done by many authors and for many subspaces of $\F(2)$, see \cite{Zoladek94} and references therein for an overview of this. Following \cite{IlyashenkoYakovenko} we consider the following normal form of quadratic systems with a Morse center at the origin
\begin{equation}
 \left\{ 
 \begin{array}{ccc}
 \dot x     & = & -ix+Ax^2+Bxy+Cy^2, \\
  \dot y    &  =   & iy+C'y^2+B'xy+A'x^2,
 \end{array}
 \ A,B,C,A',B',C' \in\C.
\right. 
\end{equation}
The Bautin ideal of the above system  has been extensively studied in the literature, see \cite{Zoladek94,IlyashenkoYakovenko}. The Bautin ideal associated  is generated by $ g_2,g_3,g_4 $, where  
\begin{eqnarray*}
g_2&:=&AB-A'B',\\  
g_3&:=& (2A+B')(A-2B')CB'-(2A'+B)(A'-2B)C'B,\\
g_4&:=& (BB'-CC')((2A+B')B'^2C-(2A'+B)B^2C'). 
\end{eqnarray*}
The computation of the primary decomposition of this ideal implies four reduced components which are explicitly written in \cite[Theorem 1]{Zoladek94}: 
\begin{enumerate}
    \item Lotdka-Volterra component ${\mathcal L}(1^3)$: $B=B'=0$,
    \item Hamiltonian ${\mathcal L}(3)$: $2A+B'=2A'+B=0$,
    \item Reversible ${\mathcal L}(1,2)$: 
    $AB-A'B'=B'^3C-B^3C'=AB'^2C-A'B^2C'=A^2B'C-A'^2BC'=A^3C-A'^3C'$, 
    \item Exceptional: $A-2B'=A'-2B'=CC'-BB'=0$. 
\end{enumerate}  
Note that the ideal of the reversible component is radical and is written in a Groebner basis (in contrast to 
\cite[section 13]{IlyashenkoYakovenko} where the corresponding "symmetric" component turns out to be  reducible).
We can also compute the ideal of its singular set. It is clear that the Hamiltonian and Lotka-Volterra components are smooth and the exceptional component has an isolated singularity at $A=\cdots=C'=0$. The reversible component has more interesting singularities: 
$$
{\rm Sing}({\mathcal L}(1,2))=\{B=B'=A=A'=0\}
={\mathcal L}(3)\cap {\mathcal L}(1^3).
$$
The foliation $\F$ with $A=B=A'=B'=0,C=C'=1$  has the first integral $f:=(\frac{1}{2}-x)(y^2-\frac{1}{3}(x+1)^2)$,  see \cite[page 159]{Iliev98}. For the computer codes used in this computation see the latex file of the present article in arxiv. 
It must be noted that ${\mathcal L}(1^3)$ itself is not smooth, for instance, it has a nodal singularity at the foliation with the first integral
$\frac{x^2+y^2}{2y-1}$ which has been studied in \cite{FrancoiseGavrilov2020}. 
\begin{proposition}
 The singular set of ${\mathcal L}(1^3)$ is  the orbit of the affine group $\Aff(\C^2)$ on the foliation with the first integral $\frac{(x+1)(y-1)}{xy}$.  
 \end{proposition}
 For an illustration of the above phenomenon see \cite[Figure 2]{FrancoiseGavrilov2020}.
\begin{proof}
We know that the kernel of the derivation of the parametrization $\tau$ in \eqref{09052022ivan} has constant dimension. This implies that all singularities of ${\mathcal L}(1^3)$ are due to the noninjectivity of $\tau$. For a foliation $\F=\F(\omega)\in {\rm Sing}({\mathcal L}(1^3))$ we get $f=l_1l_2l_3$ and $g=\tilde l_1\tilde l_2\tilde l_3$, where $\{l_i=0\}$'s (resp.  $\{\tilde l_i=0\}$'s) are distinct lines, such that $d(\frac{\omega}{f})=d(\frac{\omega}{g})=0$, and hence, $F:=\frac{f}{g}$ is a first integral of $\F$. It turns out that one of the lines $\{l_i=0\}$'s must be equal to one of $\{\tilde l_i=0\}$'s, and since $\F$ is of degree $2$,  $F$ is the quotient of two lines by another two lines. Further, $F-1$ is the quotient of a line with another two lines. We conlude that up to the action of $\Aff(\C^2)$, the foliation $\F$ has the first integral $F:=\frac{(x+1)(y-1)}{xy}$. Note that two branches of ${\mathcal L}(1^3)$ near $\F(\omega)$ correspond to $F-1=\frac{-x+y-1}{xy}$ and $\frac{F-1}{F}=\frac{-x+y-1}{(x+1)(y-1)}$. 
\end{proof}

\def\cprime{$'$} \def\cprime{$'$} \def\cprime{$'$} \def\cprime{$'$}

\end{document}